\documentclass[leqno,11pt]{amsart}
\usepackage [utf8]{inputenc}
\usepackage{amsmath}
\usepackage{amsthm}
\usepackage{amsfonts}
\usepackage{times}
\usepackage{geometry}
\usepackage{xfrac}
\usepackage{verbatim}
\usepackage{dsfont}

\usepackage[charter]{mathdesign}

\usepackage{enumerate}
\usepackage{verbatim}

\newcommand{\barint}{
         \rule[.036in]{.12in}{.009in}\kern-.16in
          \displaystyle\int  }
\def\R{{\mathbb{R}}}
\def\r{{\mathbb{R}}}

\newcommand{\RS}{\mathrm{R}}
\newcommand{\BS}{\mathfrak{B}}

\newtheorem{theo}{\bf Theorem} 
\newtheorem{coro}{\bf Corollary}[section]
\newtheorem{lem}{\bf Lemma}[section]
\newtheorem{rem}{\bf Remark}[section]

\newtheorem{ex}{\bf Example}[section]

\newtheorem{prop}{\bf Proposition}[section]

\newcommand{\wt}{\widetilde}

\newcommand{\vp}{{\varphi}}

\newcommand{\vt}{\vartheta}
\newcommand{\dv}{\mathrm{div}}
\newcommand{\ve}{\varepsilon}
\newcommand{\rn}{{\mathbb{R}^{N}}}

\newcommand{\E}{\mathcal{E}}
\newcommand{\BSD}{\BS}
\newcommand{\I}{\mathcal{I}}

\date{} 

\title[Inequalities and applications to doubly nonlinear diffusion equations]{Functional inequalities and applications\\ to doubly nonlinear diffusion equations} 

\author[Chlebicka]{Iwona Chlebicka}\address{Iwona Chlebicka\\Faculty of Mathematics, Informatics and Mechanics, University of Warsaw\\ul. Banacha 2, 02-097 Warsaw, Poland} \email{\texttt{i.chlebicka@mimuw.edu.pl}}
\author[Simonov]{Nikita Simonov}\address{Nikita Simonov\\ Ceremade, UMR CNRS n$^\circ$~7534, Universit\'e Paris-Dauphine, PSL Research University, Place de Lattre de Tassigny, 75775 Paris Cedex~16, France.}\email{\texttt{simonov@ceremade.dauphine.fr}}
\thanks{The research of I. Chlebicka was partially funded by NCN Grant 2019/34/E/ST1/00120. N.~Simonov was partially supported by the Spanish Ministry of Science and Innovation, through the  FPI-grant BES-2015-072962, associated to the project MTM2014-52240-P (Ministry of Science and Innovation, Spain), by the project MTM2017-85757-P (Ministry of Science and Innovation, Spain), by the E.U. H2020 MSCA programme, grant agreement 777822, by the Project EFI (ANR-17-CE40-0030) of the French National Research Agency (ANR), and by the DIM Math-Innov of the Region \^{I}le-de-France.}

\begin{document}

\maketitle

\begin{abstract}

We study weighted inequalities of Hardy and Hardy--Poincar\'e type and find necessary and sufficient conditions on the weights so that the considered inequalities hold.  
Examples with the optimal constants are shown. Such inequalities are then used to quantify the convergence rate of solutions to doubly nonlinear fast diffusion equation towards the Barenblatt profile.
\end{abstract}


\section{Introduction}

We investigate functional inequalities of Hardy and Poincar\'e type. Our main objective is to provide new constructive methods of derivation of such inequalities with explicit constants that are optimal in some cases. Our study starts with finding necessary and sufficient conditions on the weights so that the following \emph{Poincar\'e inequality} on the real line is valid for any compactly supported $f\in W^{1,\infty}(\R)$
\[\label{P}\tag{P}
\int_{\R} |f-(f)_{w_1}|^q w_1(s)\, ds \le C_P \int_{\R} |f'|^q w_2(s)\, ds\,, \qquad\text{where $1<q<\infty$,}
\]
 $(f)_{w_1}$ is a weighted average with respect to $w_1$, and $C_P=C_P(q,w_1,w_2)>0$, see Theorem~\ref{theo:Poin-line-const}. 
Having \emph{Poincar\'e inequality} on the real line,  we prove some compatibility conditions to be necessary and sufficient for \emph{Hardy--Poincar\'e inequality on $\rn$}, $N>2$, of a form
\[\tag{HP}\label{HP}
\int_\rn \ |\vp-\overline{\vp}|^q   w_1(x)\,dx\leq C_{HP} \int_\rn \ |\nabla \vp|^q  w_2(x)\,dx, \qquad\text{where $1<q<N$,}\]
to hold for all compactly supported $\vp\in W^{1,\infty}(\rn)$, where $\overline \vp=\frac{1}{H_1}\int_\rn \vp w_1(x)\,dx$, and $C_{HP}=C_{HP}(q,w_1,w_2)>0$, see Theorem~\ref{theo:HP-q}. Our next result is Theorem~\ref{theo:inq-sgn-Lap-q} providing a constructive PDE-based method of obtaining admissible weights  to the following \emph{Hardy inequality}
\[\tag{H}\label{H}
\int_\rn \ |\vp|^q   w_1(x)\,dx\leq C_{H} \int_\rn \ |\nabla \vp|^q  w_2 (x)\,dx, \qquad\text{where $1<q<\infty$,}\]
 for every compactly supported $\vp\in W^{1,\infty}(\rn)$ and $C_H=C_H(q,w_1,w_2)>0$.
 
The name \emph{Hardy-Poincar\'e inequalities} has been introduced in~\cite{bbdgv1}. In that paper the authors consider a family of inequalities which interpolates between the classical Hardy and Poincar\'e inequalities, see~\cite[Proposition 5]{bbdgv1}.  Since then, the name Hardy-Poincar\'e has been popularized and several authors used it, see for instance~\cite{bbdgv2,Agueh2009,Dolbeault2011, Dolbeault2013,bonforte2021stability,S2}. In the present paper, we preferred to distinguish between two different situations: when the left-hand side weight is integrable (Hardy-Poincar\'e inequality~\eqref{HP}) and when it is not true (Hardy inequality~\eqref{H}).
\newline{}

The investigation of the inequalities mentioned above is motivated by the study of nonlinear diffusion equations where they play a critical role in understanding the~asymptotic behaviour of solutions to Cauchy problems. 
The constants in the inequalities dictate the convergence rates of a~solution towards a~self-similar one. 
We apply such an inequality to investigate the long-time asymptotics of solutions to the Cauchy problem
\begin{equation}\label{dnle}
u_t=\mathrm{div}\left(|\nabla u^m|^{p-2}\,\nabla u^m\right)\,,\quad u(0,x)=u_0(x)\,.
\end{equation}

We point out that the inequalities~\eqref{HP} and~\eqref{P} retrieve the Hardy--Poincar\'e and Hardy inequalities applied in~\cite{bbdgv2,Bonforte2009} in order to obtain precise rates of convergence of solutions to~\eqref{dnle} for $p=2$ and $m<1$. Inequalities \eqref{P}, \eqref{HP}, and \eqref{H} are far better understood for $q=2$ or in bounded domains, cf.~\cite{sob-log-overview,BR,Bobkov99,Ghoussoub2013}. We are interested in providing them for general $q$ and under possibly general assumptions on the weights. 
Despite our ultimate goal is to provide a handy tool for analysis of an evolution equation,  let us mention that such inequalities play an important role in other branches of analysis. Hardy-type inequalities are used in functional analysis, probability theory, interpolation theory, and PDEs \cite{BaGo,BR,bbdgv2,ChZa,Chua,GaPa,GuW,Huang2021,KPP,PoMi,VaZu}. We refer to \cite{sob-log-overview,BR,Bobkov99,Miclo} for information on the relation between~\eqref{H},~\eqref{HP} and Log-Sobolev inequalities. The latter are of great interest in geometry in finite and infinite dimension, probability theory, and statistical mechanics.\newline{}

In order to present our first result let us define the \emph{median} of the measure $w_1(s)\,ds$ being a~point $\eta\in\R$ such that $\int_{-\infty}^\eta w_1(s)\,ds=\|w_1\|_{L^1(\R)}/2$. For any $m\in\R$ we define
\begin{equation}\begin{split}
\label{Muckenhoupt.quantities}
B_m^{+}:=\sup_{t\ge m}\left\{\int_{t}^\infty w_1(s)\,ds\,\left(\int_{m}^t w_2^\frac{1}{1-q}(s) \,ds\right)^{q-1} \right\},\\  B_m^{-}:=\sup_{t\le m}\left\{\int_{-\infty}^t w_1(s)\,ds\,\left(\int_{t}^m w_2^\frac{1}{1-q}(s) \,ds \right)^{q-1}\right\}. 
\end{split}\end{equation}
The following theorem characterizes the measures which satisfy inequality~\eqref{P}. 
\begin{theo}[Poincar\'e inequality on $\R$]\label{theo:Poin-line-const}
Let $1<q<\infty$, $0\le w_1,w_2$, $\|w_1\|_{L^1(\R)}<\infty$, and let $\eta$ be a median of the measure $w_1(s)\,ds$. Then inequality~\eqref{P} holds if and only if $B_{\eta}^{+}, B_{\eta}^{-} <\infty$. Moreover, the optimal constant $C_P$ satisfies
\[
\frac{(2^\frac{q-1}{q}-1)^q}{2^{q-1}}\, \max\left(B_\eta^+,B_\eta^-\right) \le C_P \le \frac {(2q)^q}{(q-1)^{q-1}}\, \max\left(B_\eta^+,B_\eta^-\right)
\]
\end{theo}
\noindent In the case $q=2$, the result was proven in~\cite[Théorème 6.2.2]{sob-log-overview} and~\cite{Bobkov99}. See also~\cite{Miclo} for a~discussion on the bounds on the optimal constant if $q=2$. The method of the proof is inspired by the above mentioned papers. 
\newline{}

As an application of Theorem~\ref{theo:Poin-line-const}  we prove our next result -- sufficient and necessary conditions for inequality~\eqref{HP} to hold. In order to state our result we define for any $m>0$ and for any function $0\le h\in L^1([0, \infty))$ the quantity
\begin{equation}\begin{split}\label{H2-finite}
H_2(m) :=&\max\left\{\sup_{t>m}\left(\int_t^\infty r^{N-1}h(r)\,dr\left[\int_m^t {(r^{N-1+q}h(r))^{-\frac{1}{q-1}}dr}\right]^{q-1}\right),\right.\\
&\quad\quad\ \ \left.\sup_{t\in(0,m)}\left(\int_0^t r^{N-1}h(r)\,dr\left[\int_t^m {(r^{N-1+q}h(r))^{-\frac{1}{q-1}}dr}\right]^{q-1}\right)\right\}.\end{split}
\end{equation}
With a slight abuse of notation, for a nonnegative, radial $N$-dimensional measure $h(r)r^{N-1}dr$ we define a \emph{median} to be $\eta\in[0, \infty)$ such that $\int_{0}^\eta h(r)r^{N-1}dr=\|h\|_{L^1(\R^N)}/2$. Our result reads as follows.
\begin{theo}[General Hardy--Poincar\'e inequality]\label{theo:HP-q}  Let $1<q\le 2$ if $N=2$ and $1<q<N$ if $N\ge3$. Assume $h:[0,\infty)\to[0,\infty)$ is such that
\begin{equation}
    \label{H1-finite}
    H_1:=\|h\|_{L^1(\R^N)}<\infty\,,
\end{equation}
and let $\eta$ be a median of the measure $ h(r)r^{N-1}\,dr$, $w_1(x)=h(|x|)$, and $w_2(x)=|x|^qh(|x|)$.\\
Then~inequality\eqref{HP} holds if and only if  $H_2(\eta)<\infty$. Moreover, the optimal constant $C_{HP}$  satisfies 
\begin{equation}
    \label{CHP-est}
\frac{(2^\frac{q-1}{q}-1)^q}{2^{q-1}}\,H_2(\eta)\le C_{HP}\le \,\frac{(2q)^q}{(q-1)^{q-1}} \,H_2(\eta).
\end{equation}
\end{theo}

Let us concentrate on inequalities of Hardy type~\eqref{H}. It is known that such inequalities can be proven with weights depending on solutions to elliptic PDE, see~\cite{BFT2004,DA2,DA,DA3,GM,S1,S2}. The principal idea is to re-interpret  certain versions of Caccioppoli estimate for a solution as Hardy inequalities for test functions. In our result weights $w_1,w_2$ in~\eqref{H} are determined by the means of the gradient and the $\theta$-Laplacian of a sub- or super-$\theta$-harmonic function. Let us define the operator \begin{equation}
    \label{th-Lap}\Delta_\theta g=\dv\big(|\nabla g|^{\theta-2}\nabla g\big)\qquad\text{ for } \ \theta>1.
\end{equation}
Note that for $\theta=2,$ we have $\Delta_\theta g=\Delta g$. 

\begin{theo}[General Hardy inequality]\label{theo:inq-sgn-Lap-q}
 Suppose $1<q,\theta<\infty$, $g\in W^{1,1}_{loc}(\rn)$ is positive and such that $ \Delta_\theta g \in L_{loc}^1(\rn)$ has constant sign. Then for every  compactly supported $\vp\in W^{1,\infty}(\rn)$ inequality~\eqref{H}   holds with $w_1(x)=|\Delta_\theta g|$, $w_2(x)= {|\nabla g|^{q(\theta-1)}}{|\Delta_\theta g|^{1-q}}$, and  $C_H=q^q$.
\end{theo}
\noindent The above theorem has surprisingly strong consequences as for how easy is its proof. See \cite[Theorem~2.5]{DA} and \cite[Theorem~3.1]{S1} for our inspiration. The precision of the reasoning is illustrated by in Section~\ref{sec:ex-h} with a family of inequalities, where the constant is proven to be optimal.\newline{}

As an application of Theorem~\ref{theo:HP-q} we study the \emph{asymptotic behaviour} of solutions to doubly nonlinear equation~\eqref{dnle}.  Equations as~\eqref{dnle} have been investigated since the 70's due to their intrinsic mathematical difficulties and the wide range of applications, for instance in glaciology and non-Newtonian fluids~\cite{a1970,hutter97,lad69}. For more details on cornerstones of the field we refer to~\cite{Kalashnikov1987, Vaz17} and also the monographs~\cite{Vazquez2007,JLVSmoothing,Zhuoqun2001}. Special cases of equation~\eqref{dnle} are the heat equation ($p=2$, $m=1$), the porous medium equation ($p=2$, $m>1$), and the fast diffusion equation ($p=2$, $m<1$). In the case $m=1$,~\eqref{dnle} involves the already classical $p$-Laplace operator $\Delta_p u=\mathrm{div}\left(|\nabla u|^{p-2}\,\nabla u\right)$. What is more, then evolution equation~\eqref{dnle} requires different methods for $1<p<2$ and $2<p$ being called singular and degenerate, respectively. Studies on solutions to equations like~\eqref{dnle} attract deep attention of various groups developing their theory from different points of view~\cite{Agueh2005,Agueh2009,BDM2018,Duzgun19,fornaro21,FSV2015,IMY1994,Li2001,S2020,Vespri2020}. The issue of convergence of solutions to a self-similar profile to nonlinear diffusion equations has been studied e.g. in~\cite{bbdgv2,CLM2002,Denzler2005,delpino2002a,JLVSmoothing} and in the doubly nonlinear evolution equation in~\cite{Agueh2003,Agueh2009}.

We are interested in~\eqref{dnle} with nonnegative initial datum with $u_0\in\mathrm{L}^1(\R^N)$. Our result is proven for $p>1$, $m>0$, and
\begin{equation}\label{good.range}
\frac{N-p}{p}< m(p-1)< \frac{N-p+1}{N}\,,
\end{equation}
which fits in  the \emph{fast diffusion range} where the mass of the solution is conserved along time. Under such regime infinite speed of propagation holds, i.e., if $u_0\ge0$ then $u(t)>0$ for all times $t>0$. Within~\eqref{good.range} the asymptotic behaviour of a solution is controlled by a \emph{self-similar} one called the \emph{Barenblatt} solution
\begin{equation}\label{barenblatt.solution}
B_D(t,x)=\RS(t)^{-N}\,\BS_D\left(x/\RS(t)\right)\,,
\end{equation}
where
\begin{equation}\label{barenblatt.profile}
\RS(t):=\left(1+\vartheta\,t\right)^\frac{1}{\vartheta}\,\quad\mbox{and}\quad \BS_D(x)=\left(D+\frac{1-m(p-1)}{mp}|x|^\frac{p}{p-1}\right)^\frac{p-1}{m(p-1)-1}\,.
\end{equation}
with $\vartheta=p-N[1-m(p-1)]>0$. The parameter $D$ depends on the mass of $\BS_D$.  In particular, it is known that solutions to~\eqref{dnle}, within the range~\eqref{good.range} converge to $B_D$ in the $L^1(\R^N)$-topology as $t\rightarrow\infty$.  Here, we provide an effective alternative way for one important step of the proof of~\cite[Theorem 1.1]{Agueh2009} and supply it with an estimate on a rate of convergence. 
\begin{theo}\label{theo:as} Let $N\ge3$,  $m>0$, $p>1$ as in~\eqref{good.range} and let nonnegative $ u_0\in\mathrm{L^1}(\rn)$  be such that  
\[
\BS_{D_0}(y)\le u_0(y)\le \BS_{D_1}(y)
\]
for some $D_0,D_1>0$. Let $D_\star>0$ be such that $\int_{\R^d}u_0\,dx=\int_{\R^d}B_{D_\star}(0,x)\,dx$. Then there exist
nonnegative constants $C$ and $\lambda$, and a time $t_0$ such that, for any solution to~\eqref{dnle} with initial datum $u_0$, the following holds
\begin{equation}\label{convergence-rate}
\|u(t) - B_{D_\star}(t)\|_{\mathrm{L}^1(\R^d)}\le C \,t^{-\lambda/2}\,,\quad\forall t\ge t_0\,.
\end{equation}
\end{theo}
\noindent In~\cite{Agueh2009} the authors apply the entropy method, where a Hardy--Poincar\'e inequality \eqref{HP} is used to compare the linear entropy with a linearised version version of it. In order to prove such an inequality they invoke the Persson's Theorem as it was done in~\cite{bbdgv1}, which one may avoid entirely. Our main contribution is to prove the needed Hardy--Poincar\'e inequality by the constructive method of Theorem~\ref{theo:HP-q}.
\newline

Our last result is motivated by study of the asymptotic behaviour of the fast diffusion range of the \emph{evolution p-Laplace equation}, i.e. equation~\eqref{dnle} with $m=1$ and $1<p<2$. Only few results are available in this range, see~\cite{Agueh2008,Bonforte2021}. Therefore, in a forthcoming paper, we will give a more detailed analysis for a natural class of initial data. In this work we provide the relevant Hardy--Poincar\'e inequality for this range of parameters. For the proof see Section~\ref{sec:coro}.

\begin{coro}\label{coro:H-diff}
Let $1<p<2$, $N>2$, $N>\tfrac{p}{(2-p)(p-1)}$. 
Then for all compactly supported $\vp\in W^{1,\infty}(\rn)$ it holds that
 \begin{equation}
    \label{inq:ex-h-diff}
\int_\rn \ |\vp|^2|x|^{-\tfrac{p}{p-1}}\left(1+|x|^{\tfrac{p}{p-1}}\right)^{-\frac{p-1}{2-p}} \,dx\leq C_{H} \int_\rn \ |\nabla \vp|^2|x|^{-\frac{2-p}{p-1}}\left(1+|x|^{\tfrac{p}{p-1}}\right)^{-\frac{p-1}{2-p}} \,dx\end{equation} 
with a positive finite constant $C_H$. If we additionally assume that $N\leq 7$ {\bf OR} $p\not\in (p_-,p_+)$ for
\[
 p_- = \tfrac{3}2  -\tfrac{1}{2} \sqrt{\tfrac{N - 7}{N + 1}} \quad \text{and}\quad p_+ = \tfrac{3}2  +\tfrac{1}{2} \sqrt{\tfrac{N - 7}{N + 1}}\, ,
\]
then the constant $C_H$ is optimal and reads
\[C_{H}=4\big(N-\tfrac{p}{(2-p)(p-1)}\big)^{-2}\,.\]

\end{coro}

The range $N<\tfrac{p}{(2-p)(p-1)}$ is not considered in the above result since the left-hand-side weight becomes integrable, so different techniques are needed in order to obtain an inequality, see Theorem~\ref{theo:HP-q}. From the point of view of asymptotics, in this last range the difference of two Barenblatt profiles is integrable, while in the range considered in the above corollary is not. The case $N=\tfrac{p}{(2-p)(p-1)}$ happens only if $N\ge6$ and for $p=\tfrac32-\tfrac1{2N}\pm\tfrac{\sqrt{N^2-6N+1}}{2N}$. These two last exponents play the same role played by $m_\star=(N-4)/(N-2)$ in the better understood case of the fast diffusion equation (eq.~\eqref{dnle} with $p=2$, $0<m<1$), where particularly sophisticated techniques were needed to deal with this strange exponent, see~\cite{bbdgv2,Bonforte2009}.

\medskip

\noindent \emph{Organization of the paper. } Section~\ref{sec:inq} concerns proofs of inequalities~\eqref{P}, \eqref{HP}, and~\eqref{H}. In Section~\ref{sec:ex-hp} we show examples of \eqref{HP} and in  Section~\ref{sec:ex-h}  of \eqref{H}; here Corollary~\ref{coro:H-diff} is proven. The proof of Theorem~\ref{theo:as} is presented in Section~\ref{sec:as}.

\section{Functional inequalities}\label{sec:inq}

\subsection{Poincar\'e inequality on $\r$. Proof of Theorem~\ref{theo:Poin-line-const}} 
\label{sec:Poin-line-const}We will employ the classical result of Muckenhoupt.
\begin{lem}[Muckenhoupt's inequality, Theorem~1, \cite{M}]\label{lem:Mw} Let $1< q<\infty$, $0\le w_1,w_2$  such that $w_1\in L^1_{loc}([0,\infty))$. There exists $0 < C_{M}< \infty$ for which \[\int_0^\infty |f-f(0)|^q w_1(r)\, dr \leq C_{M}\int_0^\infty |f'|^q w_2(r)\,dr\] is true for every compactly supported $f\in W^{1,\infty}([0,\infty))$ if and only if
\begin{align}\label{HM-w_1-finite}
    H_{M}:=&\sup_{\rho>0}\left(\int_\rho^\infty w_1(r)\,dr \left[\int_0^\rho {(w_2(r))^{-\frac{1}{q-1}}dr}\right]^{q-1}\right)<\infty\nonumber
\end{align}
Moreover, then the optimal constant $C_M$ satisfies 
$H_{M}\leq C_{M}\leq {q^q}{(q-1)^{1-q}}H_{M}.$ 
\end{lem}

We are in a position to prove Poincar\'e inequality~\eqref{P} on the real line.
\begin{proof}[Proof of Theorem~\ref{theo:Poin-line-const}]
We assume first that $B_\eta^+$ and $B_\eta^-$ are both finite. For any $m$ we have that
\begin{align*}
    \int_{-\infty}^\infty |f-(f)_{w_1}|^q w_1(s)\, ds&\leq 2^{q}\int_{-\infty}^\infty |f-f(m)|^q w_1(s)\, ds\\
    &= 2^{q}\int_{-\infty}^m |f-f(m)|^q w_1(s)\, ds+2^{q}\int_{m}^\infty |f-f(m)|^q w_1(s)\, ds\\
    &=:2^q(L_1(m)+L_2(m)).
\end{align*}
Let us define
\[
J_1(m):=\int_{-\infty}^m |f'|^q w_2(s)\,ds\quad\mbox{and}\quad J_2(m):=\int_{m}^\infty|f'|^q w_2(s)\,ds\,,
\]
and let us call $A_m^-$ ($A_m^+$ resp.) the optimal constant of the inequality $ L_1(m)\le A_m^- J_1(m) $  (of inequality $ L_2(m)\le A_m^+ J_2(m)$ resp.). In particular, if we consider $m=\eta$,  we obtain that 
\[
B_\eta^- \le  A_\eta^- \le {q}{(q')^{q-1}} \,B_\eta^-\quad\mbox{and}\quad B_\eta^+ \le  A_\eta^+ \le {q}{(q')^{q-1}} \,B_\eta^+
\]
as a consequence of a change of variables and Lemma~\ref{lem:Mw}. By summing up the previous inequalities, we find that
\begin{equation*}\begin{split}
 \int_{-\infty}^\infty |f-(f)_{w_1}|^q w_1(s)\, ds &\le  2^q\,(L_1(m)+L_2(m)) \\
 &\le 2^q\,\max\{A_\eta^+\,A_\eta^-\}  (J_1(m)+J_2(m)) \\
 & \le 2^q\, {q}{(q')^{q-1}}\, \max\left(B_\eta^+,B_\eta^-\right)\,\int_{\R} |f'|^q w_2(s) \,ds\,,
\end{split}
\end{equation*}
which proves inequality~\eqref{P} with the wanted upper bound on the constant $C_P$.\newline{}

Assume now that inequality~\eqref{P} holds with a finite constant $C_P>0$. Let us first restrict our attention to the case when $B_\eta^+, B_\eta^-<\infty$ (so that, by the optimality of $A_\eta^+$ and $A_\eta^-$, $\infty>A_\eta^+\ge B_\eta^+$ and $\infty>A_\eta^-\ge B_\eta^-$ ) and will to prove the lower bound on the constant $C_P$.  Without any loss of generality we can assume that $\|w_1\|_{L^1(\R)}=1$.
Let us recall that $A_\eta^+$ is the optimal constant for inequality $ L_2(\eta)\le A_\eta^+ J_2(\eta)$. By optimality, for any $\varepsilon>0$ there exists a function $f_\varepsilon$ and such that
\begin{equation}\label{inq.P.1}
\int_\eta^{\infty} \left|\int_\eta^s f_\varepsilon(t)dt\right|^q w_1(s)\,ds \ge (A_\eta^+-\varepsilon) \int_\eta^\infty |f_\varepsilon(s)|^q w_2(s) ds\,.
\end{equation}
Note that, without loss of generality, we can assume that $f_\ve\geq 0$. Let us define $F_\varepsilon(x)=0$ if $x\le \eta$ and $F_\varepsilon(x)=\int_\eta^x f_\varepsilon(t)dt$ if $x\ge \eta$. Therefore, by the property of the median, we have ${w_1}(\{F_\varepsilon>0\})\le1/2$. Then, by H\"older inequality, we obtain
\[
\int_{\R}F_\varepsilon(t) {w_1}(t)\, dt \le
\left(\int_{\R} |F_\varepsilon(t)|^q\,{w_1}(t) dt\right)^\frac{1}{q}\, {w_1}\left(\{F_\varepsilon>0\}\right)^\frac{q-1}{q}\le 2^{-\frac{q-1}{q}}{\left(\int_{\R} |F_\varepsilon(t)|^q\,{w_1}(t)\, dt\right)^\frac{1}{q}} \,.
\]
Then, by using the above inequality with the triangle inequality and inequality~\eqref{inq.P.1}, we find
\begin{equation*}\begin{split}
\int_{\R} |F_\varepsilon(t)-(F_\varepsilon)_{w_1}|^q\,{w_1}(t) dt
\ge &\left(\left(\int_{\R} |F_\varepsilon(t)|^q\,{w_1}(t) dt\right)^\frac{1}{q}-(F_\varepsilon)_{w_1}\right)^q\\
\ge & \frac{(2^\frac{q-1}{q}-1)^q}{2^{q-1}}\,\int_{\R} |F_\varepsilon(t)|^q\,{w_1}(t) dt \\
\ge &\frac{(2^\frac{q-1}{q}-1)^q}{2^{q-1}}\, (A_\eta^+-\varepsilon)\, \int_\eta^\infty |f_\varepsilon(s)|^q {w_2}(s)\,ds\, \\
\ge &\frac{(2^\frac{q-1}{q}-1)^q}{2^{q-1}}\, (B_\eta^+-\varepsilon)\,  \int_\eta^\infty |f_\varepsilon(s)|^q {w_2}(s)\,ds\, \\
\ge & \frac{(2^\frac{q-1}{q}-1)^q}{2^{q-1}}\, \frac{(B_\eta^+-\varepsilon)}{C_P}\,\int_{\R} |F_\varepsilon(t)-(F_\varepsilon)_{w_1}|^q\,{w_1}(t) dt\,,
\end{split}
\end{equation*}
where we used that $A_\eta^+$ is the optimal constant and inequality~\eqref{P}. This proves that
\[C_P\ge \frac{(2^\frac{q-1}{q}-1)^q}{2^{q-1}}\, (B_\eta^+-\varepsilon)\]
for any $\varepsilon>0$. The same construction can be used for the case of $B_\eta^-$. 
It remains to prove that if inequality~\eqref{P} holds for a finite $C_P>0$ then $B_\eta^+,B_\eta^-<\infty$.  To justify this one may argue by contradiction using the same lines as above. 
\end{proof}

When we apply Theorem~\ref{theo:Poin-line-const} for $w_1(r)=r^{N-1}h(r)\mathds{1}_{[0,\infty)}(r)$ and $w_2(r)=r^{N-1+q}h(r)\mathds{1}_{[0,\infty)}(r)$, and denote $(f)=(f)_{w_1},$ we get the following consequence.
\begin{coro}\label{coro:P}
Suppose $1< q<\infty$ and $0\le h\in L^1_{loc}([0,\infty))$ is such that~
and there exists $m>0$ for which $H_2(m)$ given by~\eqref{H2-finite} is finite. There exists $0 < C< \infty$ such that
\begin{equation}
    \label{inq:better-w_1}\int_0^\infty |f-(f)|^q r^{N-1} h(r)\, dr\leq C  \int_0^{\infty} |f'|^q r^{N-1+q}h(r)\, dr
\end{equation}
holds for every compactly supported $f\in W^{1,\infty}([0,\infty))$. Moreover, $C\leq 2^q\, {q}{(q')^{q-1}}\, H_2(m)$.
\end{coro}

\subsection{Hardy--Poincar\'e inequality on $\rn$}\label{sec:proof-HP}
We prove the inequality with radial weights as an application of the Poincar\'e inequality~\eqref{P}. For this we introduce the standard change of variables from Cartesian to spherical coordinates, that is \[r = |x|\qquad\text{ and }\qquad\vt = x/|x|.\] In these coordinates, the gradient can be written as $(\partial r, \frac{1}{r}\nabla_\vt)$ where $\partial r=\frac{x}{r}\cdot \nabla $ is the partial derivative with respect to the radial variable $r$ and $\nabla_\vt$ is the derivative with respect to the angular variables. Then\begin{equation}
     \label{rad-grad}|\partial_r{f}(r,\vt) |^2+\tfrac{1}{r^2}|\nabla_\vt{f}(r,\vt)|^2=|\nabla f(x)|^2.
 \end{equation}
 
 By $S^{N-1}\subset\rn$ we denote the unit sphere and parametrize it with the variable~$\vt$. 
Moreover, for every  compactly supported function  $f\in W^{1,\infty}(\rn)$ we denote the directional average as
\begin{equation}
    \label{wtfmu} \wt{f}(\vt):=\frac{|S^{N-1}|}{H_1}\int_0^\infty f(r,\vt){r^{N-1}}h(r)\,dr.
\end{equation}
Recall that the global average with respect to $h$ is given by
\begin{equation}
    \label{barf}
\overline{f}=\frac{1}{|S^{N-1}|}\int_{S^{N-1}}\wt{f}(\vt)d\vt=\frac{1}{H_1 }\int_{\rn} f(x)h(|x|)\,dx.
\end{equation}
We recall the Poincar\'e inequality on the sphere.
\begin{lem}[Spherical Poincar\'e inequality]
\label{lem:sph-Poin} Let $1\leq q < N$ if $N\geq 3,$ and $1\le q \le 2$ when $N=2$. Then there exists $C_{\rm sph}>0$ such that \begin{equation}
    \label{inq:sph-Poin}
\int_{S^{N-1}} |\wt{f}(\vt)-\overline{f}|^q \,d\vt\leq C_{\rm sph}
\int_{S^{N-1}} |\nabla_\vt \wt{f}(\vt) |^q \,d\vt\end{equation}
for every $f\in W^{1,q}(S^{N-1})$.
\end{lem}
\noindent In the case $N\ge3$, inequality~\eqref{inq:sph-Poin} is a direct consequence of \cite[Theorem~2.10]{hebey} formulated for a smooth and compact Riemannian manifold. In the case $q=2$ and $N=2$, it follows from a decomposition in spherical harmonics which then coincides with the Fourier basis, see~\cite{andrews99}. To our best knowledge its optimal value of the constant from~\eqref{inq:sph-Poin} is  known only in the case of $q=2$, when $C_{\rm sph}=\frac{1}{N-1}$, see~\cite[Chapter~4, Proposition~1]{sph-Poin-constant-book}.  \newline{}

The following result would be instrumental in the proof of Theorem~\ref{theo:HP-q}.
\begin{prop}\label{prop:radsuff} Suppose assumptions of Theorem~\ref{theo:HP-q} are satisfied. If for every  compactly supported function  $f\in W^{1,\infty}(\rn)$ the following inequality holds
\begin{equation}\label{eq:radinq-q}
\int_0^\infty\int_{S^{N-1}} |f(r,\vt)-\wt{f}(\vt)|^q r^{N-1}h(r)\,d\vt\, dr\leq C_1
 \int_0^\infty \int_{S^{N-1}} |\partial_r {f}(r,\vt) |^q r^{N-1+q}h(r)d\vt\, dr,
\end{equation}
 then there exist $C_{HP}>0$, such that for the same $f$ one has~\eqref{HP} with $w_1(x)=h(|x|)$, $w_2(x)=|x|^qh(|x|)$, and a positive constant $C_{HP}\leq 2^{q-1}\,\max\{C_1, H_1^q\,C_{\rm sph}\,|S^{N-1}|^{-q} \}$. 
\end{prop}
\begin{proof}By~\eqref{barf} and Jensen's inequality it holds
\[\begin{split}\int_\rn|f(x)-\overline{f}|^qh(|x|)\,dx&= \int_0^\infty \int_{S^{N-1}} |f(r,\vt)-\overline{f}|^q r^{N-1}h(r)\,{d\vt}\,dr\\
&\leq 2^{q-1}\int_0^\infty \int_{S^{N-1}} |f(r,\vt)-\wt{f}(\vt)|^qr^{N-1}h(r)\,{d\vt}\,dr\\
&\quad+2^{q-1}\int_0^\infty \int_{S^{N-1}} |\wt{f}(\vt)-\overline{f}|^qr^{N-1}h(r)\,{d\vt}\,dr=:I_1+I_2.\end{split}\]
Since the first term on the right-hand side above can be estimated due to~\eqref{eq:radinq-q} and then by properties of the gradient~\eqref{rad-grad}, we get
\[I_1\leq 2^{q-1}C_1
 \int_0^\infty \int_{S^{N-1}} |\partial_r {f}(r,\vt) |^q r^{N-1+q}h(r)\,{d\vt}\, dr\leq2^{q-1} C_1
\int_{\rn}|\nabla {f}(x)|^q |x|^q h(|x|)\, dx\,.\]
 As for the second one we use~\eqref{H1-finite} to obtain
\[I_2=2^{q-1}\frac{H_1}{|S^{N-1}|}\int_0^\infty  |\wt{f}(\vt)-\overline{f}|^q\,{d\vt}\]
and we apply the spherical Poincar\'{e} inequality (Lemma~\ref{lem:sph-Poin}) to get
\[I_2=2^{q-1}\frac{H_1}{|S^{N-1}|}\int_{S^{N-1}} |\wt{f}(\vt)-\overline{f}|^q \,{d\vt}\leq 2^{q-1} \frac{H_1}{|S^{N-1}|} C_{\rm sph}
\int_{S^{N-1}} |\nabla_\vt \wt{f}(\vt) |^q \,{d\vt}=:J.\]
As $r^{N-1}h(r)|S^{N-1}|/H_1\,d r$ is a probability measure by Jensen's inequality we have that
\[\begin{split}J&= 2^{q-1} \frac{H_1}{|S^{N-1}|} C_{\rm sph}  \int_{S^{N-1}} \left|\nabla_\vt \int_0^\infty {f}(r,\vt)r^{N-1}h(r)\, dr\right|^q\,{d\vt}\\
&=2^{q-1} \left(\frac{H_1}{|S^{N-1}|}\right)^{q+1} C_{\rm sph} \int_{S^{N-1}} \left| \int_0^\infty \nabla_\vt{f}(r,\vt)r^{N-1}h(r)\frac{|S^{N-1}|}{H_1}\, dr\right|^q\,{d\vt}\\
&\leq 2^{q-1}\left(\frac{H_1}{|S^{N-1}|}\right)^{q} C_{\rm sph} \int_{S^{N-1}}  \int_0^\infty \left|\nabla_\vt{f}(r,\vt)\right|^qr^{N-1}h(r)\, dr\,{d\vt}=:K.\end{split}\]
Due to~\eqref{rad-grad}  we have  $r^{-q}\,|\nabla_\vt f(r, \vt)|^q \le  |\nabla f(x)|^q$ 
and we infer that
\[\begin{split}
K\leq 2^{q-1}\left(\frac{H_1}{|S^{N-1}|}\right)^{q} C_{\rm sph} \int_\rn |\nabla f(x)|^q\,|x|^{q}h(|x|)\,dx .\end{split}\]
By summing up the previous computations we have the claim.
\end{proof}

We are in the position to prove the Hardy--Poincar\'e inequality.

\begin{proof}[Proof of Theorem~\ref{theo:HP-q}] By Corollary~\ref{coro:P}  with
 \begin{equation*}
(f)=\frac{|S^{N-1}|}{H_1}\,\int_{0}^\infty  f(r)\,r^{N-1}\, h(r)\,dr,
\end{equation*} 
and   $H_1$ from~\eqref{H1-finite}, we have the following inequality \[\int_0^\infty |f-(f)|^q r^{N-1}h(r)\, dr   \leq 2^{q}C_M\int_0^\infty |f'|^q r^{N-1+q}h(r)\,dr\]
for all compactly supported $f\in W^{1,\infty}([0,\infty))$. Hence, also the following radial inequality follows
\[\int_0^\infty |f(r,\vt)-\wt{f}(\vt)|^q r^{N-1}h(r)\, dr\leq C_1  \int_0^\infty  |\partial_r {f}(r,\vt) |^q r^{N-1+q}h(r) \, dr\quad\mbox{for a.e.}\,\,\vt\in S^{N-1}\] with $C_1=2^{q}C_M$. The above inequality implies~\eqref{eq:radinq-q}, which by Proposition~\ref{prop:radsuff} completes the proof of~\eqref{HP} with $w_1(x)=h(|x|)$ and $w_2(x)=|x|^qh(|x|)$. Bounds on the constant result from the estimates on $c_P$ provided in Theorem~\ref{theo:Poin-line-const}.\end{proof}

\subsection{Hardy  inequality on $\R^N$. Proof of Theorem~\ref{theo:inq-sgn-Lap-q}}\label{sec:proof-H}

\begin{proof}[ Proof of Theorem~\ref{theo:inq-sgn-Lap-q}]
Let us   observe that integration by parts and the Cauchy-Schwartz inequality imply
\[ \int_\rn|\vp|^q|\Delta_\theta g|\,dx \leq q \int_\rn|\vp|^{q-1}|\nabla \vp||\nabla g|^{\theta-1}\,dx=:I.\]

By the H\"older inequality,  we have
\[\begin{split}
I &=q \int_\rn\left(|\vp|^{q-1}{|\Delta_\theta g|^\frac{q-1}{q}}\right)\left(\frac{|\nabla \vp||\nabla g|^{\theta-1}}{{|\Delta_\theta g|^{\frac{q-1}{q}}}}\right)\,dx\\
&\leq q\left( \int_\rn  |\vp|^q|\Delta_\theta g|dx\right)^{(q-1)/q}\left(\int_\rn |\nabla  \vp |^q \frac{|\nabla g|^{q(\theta-1)}}{|\Delta_\theta g|^{q-1}} \,dx\right)^{1/q}. \end{split}\]

Summing up the above remarks, we get
\[\left( \int_\rn |\vp|^q|\Delta_\theta g|\,dx\right)^{1/q}\leq q \left(\int_\rn |\nabla  \vp |^q \frac{|\nabla g|^{q(\theta-1)}}{|\Delta_\theta g|^{q-1}} \,dx\right)^{1/q}.\]
\end{proof}
\section{Examples}

\subsection{Hardy--Poincar\'e inequalities}\label{sec:ex-hp}
As a consequence of Theorem~\ref{theo:HP-q}, we get the following version of weighted Hardy--Poincar\'e inequality on $\rn$.
\begin{ex}\label{ex:hp}  Let $1<q\le 2$ if $N=2$ and $1<q<N$ if $N\ge3$. Assume   $\alpha<0<\beta$, $N+\gamma>0$, $N+\gamma+\alpha\beta<0$. Then there exists a finite constant  $C=C(N,\alpha,\beta,\gamma,q)>0$, such that for every compactly supported $\vp\in W^{1,\infty}(\rn)$ the following inequality holds true
 \begin{equation}
    \label{inq:ex-hp}
\int_\rn \ |\vp-\overline{\vp}|^q|x|^\gamma(1+|x|^\beta)^\alpha \,dx\leq C \int_\rn \ |\nabla \vp|^q |x|^{\gamma+q}(1+|x|^\beta)^\alpha \,dx,\end{equation}
where $\overline\vp$ is an average of $\vp$ with respect to $|x|^\gamma(1+|x|^\beta)^\alpha \,dx$.
\end{ex}
\begin{proof} Note that for\begin{equation}
    \label{h-choice}
h(|x|)=|x|^\gamma(1+|x|^\beta)^\alpha.
\end{equation}
it holds that $H_1<\infty$  since $N+\gamma>0$ and $N+\gamma+\alpha\beta<0$. We will show that $H_2(m)<\infty$ for any $m>0$. We notice that
\begin{align*}H_2(m)=\max\{A[m],B[m]\},
\end{align*} 
where
\begin{align*}
A[m] :=&\sup_{t>m} A_1(t)A_2(t)=\sup_{t>m}\left(\int_t^\infty r^{N-1+\gamma}{(1+r^{\beta})^\alpha}\,dr\ \left[\int_m^t {\Big(r^{N-1+\gamma+q}\big(1+r^{\beta}\big)^\alpha\Big)^{-\frac{1}{q-1}}dr}\right]^{q-1}\right)\\
B[m]:=&
\sup_{t\in (0,m)}\left(\int_0^t r^{N-1+\gamma}{(1+r^{\beta})^\alpha}\,dr\ \left[ \int_t^m {\Big(r^{N-1+\gamma+q}\big(1+r^{\beta}\big)^\alpha\Big)^{-\frac{1}{q-1}}dr}\right]^{q-1}\right).\nonumber
\end{align*} 
Since $A_1(t)\leq c(N,\alpha,\beta,\gamma,q) t^{N+\gamma+\alpha\beta}$ and $A_2(t)\leq c(N,\alpha,\beta,\gamma,q)(m^{-(N+\gamma+\alpha\beta)}+m^{-(N+\gamma)}),$ 
we have $A[m]<\infty$. On the other hand
\[B[m]\leq \sup_{t<m}\Big(|t^{N+\gamma}(1+m^{\beta})^{-\frac{\alpha}{q-1}}\big(t^{-\frac{N+\gamma}{q-1}}\big)^{q-1}\Big)\leq  c(\alpha,\beta,\gamma,q,m)\]
being finite under the assumed regime. In turn $H_2=\max\{A[m],B[m]\}<\infty$ for any (fixed) $m>0$.  Therefore, we can apply Theorem~\ref{theo:HP-q}  to deduce~\eqref{inq:ex-hp} with a constant depending additionally on $m$.   If one chooses $m$ to be a median of $r^{N-1+\gamma}(1+r^\beta)^\alpha\,dr$, the constant can be estimated by Theorem~\ref{theo:inq-sgn-Lap-q}. Median $\eta$ is estimated in Remark~\ref{rem:eta}.
\end{proof}

\begin{rem}\label{rem:eta}\rm  One can estimate a median $\eta$ of the weight of Example~\ref{ex:hp}, namely $|x|^\gamma(1+|x|^\beta)^\alpha \,dx$, as follows 
\[\left(\tfrac{N+\gamma }{2^{|\alpha|+1}|N+\gamma+\alpha\beta|}\right)^\frac{1}{N+\gamma}\leq\eta\leq 2^\frac{|\alpha|+1}{|N+\gamma+\alpha\beta|}.\]
For this it is enough to observe that
\[H_1\geq \int_1^\infty r^{N-1+\gamma}{(1+r^{\beta})^\alpha}\,dr\geq  2^\alpha \int_1^\infty r^{N-1+\gamma+\alpha\beta}\,dr=\frac{2^\alpha}{|N+\gamma+\alpha\beta|},\]
 \[\frac{H_1}{2}=\int_\eta^\infty  r^{N-1+\gamma}{(1+r^{\beta})^\alpha}\,dr\leq \int_\eta^\infty r^{N-1+\gamma+\alpha\beta}\,dr=\frac{\eta^{{N+\gamma+\alpha\beta}}}{|N+\gamma+\alpha\beta|},\]
 and 
 \[\frac{H_1}{2}=\int_0^\eta  r^{N-1+\gamma}{(1+r^{\beta})^\alpha}\,dr\leq \int_0^\eta r^{N-1+\gamma}\,dr=\frac{\eta^{{N+\gamma}}}{|N+\gamma|}.\]
\end{rem}

\begin{rem}\rm We are particularly interested in the special case of Example~\ref{ex:hp} with the choice
\[\gamma=0,\qquad \beta=\tfrac{p}{p-1},\qquad\text{and}\qquad \alpha=-\tfrac{1}{2-p},\]
for $1<p<2$, which find application in asymptotics of fast diffusion equation, see Section~\ref{sec:as}.
\end{rem}

\subsection{Hardy inequalities}\label{sec:ex-h}

As a consequence of Theorem~\ref{theo:inq-sgn-Lap-q} we get the following family of Hardy-type inequalities on~$\rn$.  We also refer to~\cite{Huang2021,bbdgv1} where the optimal constant of some of the inequalities in the following family were already computed.
 \begin{ex}\label{ex:2_laplace}
Let $q>1$, $\alpha,\beta,\gamma\in\R$, $\alpha<0<\beta$, $\gamma+N>0$, $|{\alpha}{\beta}+\gamma+2|\geq|\gamma+2|$, 
\begin{align}
    \label{war1}{\rm sgn}({\alpha}{\beta}+\gamma+2)={\rm sgn}(\gamma+2),\quad \text{and}\quad {\alpha}{\beta}+\gamma+N>0.
\end{align}
Then  for all compactly supported $\vp\in W^{1,\infty}(\rn)$ it holds that
 \begin{equation}
    \label{inq:ex-h}
\int_\rn \ |\vp|^q|x|^\gamma(1+|x|^\beta)^\alpha \,dx\leq C_{H} \int_\rn \ |\nabla \vp|^q |x|^{\gamma+q}(1+|x|^\beta)^\alpha \,dx\end{equation} with a positive, finite constant $C_{H}$.
If we additionally assume that 
 \begin{equation}
    \label{opt-as}\alpha\beta+2(\gamma+1)+N\leq 0,\end{equation} then the constant $C_H$ is optimal and reads
\[C_{H}= \left(\frac{q}{\alpha\beta+\gamma+N}\right)^q\,.\]

\end{ex}
\begin{proof} We will apply Theorem~\ref{theo:inq-sgn-Lap-q} with $\theta=2$. For $g(x) = |x|^{\gamma+2}(1+|x|^{\beta})^{{\alpha}}$, 
it holds
\begin{align*}
    &|\Delta g(x)|=|x|^\gamma (1+|x|^\beta)^{\alpha-2}\left|\left[|x|^{\beta}\Big((\alpha\beta+\gamma+2)(\alpha\beta+\gamma+N)\Big)+ (\gamma+2)(\gamma+N)\right](1+|x|^\beta) - |x|^\beta\alpha\beta\big(\alpha-1\big) \right|\,,\\
     &|\nabla g(x)|^{q}|\Delta  g(x)|^{1-q}=|x|^{\gamma+q}(1+|x|^\beta)^{\alpha-2+q}\big||x|^\beta(\alpha\beta+\gamma+2)+\gamma+2\big|^{q}\cdot\\
      &\cdot\left|\left[|x|^{\beta}\Big((\alpha\beta+\gamma+2)(\alpha\beta+\gamma+N)\Big)+ (\gamma+2)(\gamma+N)\right](1+|x|^\beta) - |x|^\beta\alpha\beta\big(\alpha-1\big) \right|^{1-q}\,.
\end{align*}
Note that $|\Delta_\theta g|\in L^1_{loc}(\rn)$ since $\gamma+N>0$. Furthermore, then $c_1 h(|x|)\leq 
    |\Delta g(x)|$ for \begin{align*}
    c_1 =& \inf_{|x|\geq 0 } \tfrac{ \left|\left[|x|^{\beta}\Big((\alpha\beta+\gamma+2)(\alpha\beta+\gamma+N)\Big)+ (\gamma+2)(\gamma+N)\right](|x|^\beta+1) - |x|^\beta\alpha\beta\big(\alpha-1\big) \right|}{(|x|^{\beta}+1)^2}\\
    =&\inf_{s\geq 0 }  \frac{ \left |s^2(\eta+2)(\eta+{N})+s\big((\eta+2)(\eta+{N})+(\gamma+2)(\gamma+N)-\alpha\beta(\alpha-1)\big) +(\gamma+2)\big(N+\gamma\big)\right|}{(s+1)^2}\\=:&\inf_{s\geq 0 } \mathcal {C}_1(s),
\end{align*}
where $\eta=\alpha\beta+\gamma$.
Notice that under \eqref{war1} we have $(\eta+2)(\eta+N)<0$, $(\gamma+2)(\gamma+N)<0$ and $\alpha\beta(\alpha-1)>0$. So the numerator of $\mathcal {C}_1(s)$ is separated from $0$ and $c_1$ is well defined. Under additional assumption
\[
\alpha\beta+2(\gamma+1)+N\leq 0
\]
we have 
\[c_1= \mathcal {C}_1(\infty)=(-\alpha\beta-\gamma+-2)\,(\alpha\beta+\gamma+N).\]
Indeed, let us consider 
\[
\xi(s) = \log (\mathcal {C}_1(s))
\]
and denote the numerator of $\mathcal {C}_1(s)$ by
\begin{align*}
P(s) = &s^2(-\eta-2)(\eta+{N})+s\big((-\eta-2)(\eta+{N})+(-\gamma-2)(\gamma+N)+\alpha\beta(\alpha-1)\big) +\\&+(-\gamma-2)\big(N+\gamma\big):= as^2+bs+c
\end{align*}
\begin{align*}
    \xi'(s) =\tfrac{2as+b}{P(s)}-\tfrac{2}{s+1}=\tfrac{s(2a-b)+(b-2c)}{P(s)(s+1)}.
\end{align*}
We observe that 
\begin{align*}
2a-b=& (-\eta-2)(\eta+N)-(-\gamma-2)(\gamma+N) -\alpha\beta(\alpha-1)\\
\leq& (-\eta-2)(\eta+N)-(-\gamma-2)(\gamma+N)\leq 0,
\end{align*}
where last inequality is equivalent to $\alpha\beta+2(\gamma+1)+N\leq 0$
Therefore $\xi'$ is decreasing, so $\xi$ and further $\mathcal {C}_1$ attain minimum at $0$ or at $\infty$. But
\[c_1=\min \{\mathcal {C}_1(0), \mathcal {C}_1(\infty)\}=\min\{|\gamma+2|\,|\gamma+N|,|\eta+2|\,|\eta+N|\}= (-\eta-2)\,(\eta+N).\] 
On the other hand $
    |\nabla g(x)|^{q}|\Delta g(x)|^{1-q}\leq c_2 h(|x|)|x|^q$ for
 \begin{align*}
    c_2 =& \sup_{|x|\geq 0 } \tfrac{|\nabla g(x)|^q|\Delta  g(x)|^{1-q}}{ h(|x|) |x|^q }=\sup_{|x|\geq 0 }\left(\tfrac{|\nabla g(x)|}{h(|x|)||x|}\right)^q\mathcal {C}_1(|x|)^{1-q}
\end{align*}
Consider 
\[
f(s) =\tfrac{s(-\eta-2)+(-\gamma -2)}{1+s}
\]
and
\begin{align*}
    (\log f)'(s)=&\tfrac{-\eta-2}{s(-\eta-2)+(-\gamma-2)}-\tfrac{1}{s+1}=
 \tfrac{(-\eta-2) -(-\gamma-2)}{(s+1)(s(-\eta-2)+(-\gamma-2))}\geq 0
\end{align*}
Therefore $f(s)$ is increasing and under our assumption it holds $\inf_{s\geq 0}\mathcal {C}_1(s) =\mathcal {C}_1(\infty) $
\[c_2= \sup_{s\geq 0} f(s)^q \mathcal {C}_1(s)^{1-q}= f(\infty)^q\mathcal {C}_1(\infty)^{1-q} =|\alpha\beta+\gamma+2|\,|\alpha\beta+\gamma+N|^{1-q}\]
and by Theorem~\ref{theo:inq-sgn-Lap-q} we have $C_{H}\leq q^q \tfrac{c_2}{c_1}.$ To motivate optimality of $C_H$ in the special case we compare it with the optimal constant in the classical Hardy inequality. Recall that we are in the regime when  $c_1=\mathcal {C}_1(\infty)$ and, consequently, $C_H=q^q\big(\alpha\beta+\gamma+N\big)^{-q}$. Let us consider the family of rescaled functions $\vp_s(x)=\vp(sx)\in C_c^\infty(\rn)$ with $s>0$. We multiply both sides of~\eqref{inq:ex-h-diff} by $t^{\alpha\beta+\gamma}$ and rearrange them to get
\begin{align}
    \label{l1}
\int_\rn \ |\vp_s(x)|^q\, |tx|^\gamma (t^{\beta}+|tx|^{\beta})^{\alpha} \,dx\leq C_H\int_\rn \ |\nabla \vp_s(x)|^q\,|tx|^{\gamma+q}(t^{\beta}+|tx|^{\beta})^{\alpha}\,dx\,.
\end{align}
After the change of variables $y=xt$ we obtain
\[\int_\rn \ |\vp(y)|^q\, |y|^\gamma(t^{\beta}+|y|^{\beta})^\alpha \,dy\leq C_H \int_\rn \ |\nabla \vp(y)|^q\,|y|^{\gamma+q}(t^{\beta}+|y|^{\beta})^\alpha \,dy\,.\]
Using the Lebesgue Monotone Convergence Theorem, we let $t\searrow 0$ and get
\begin{align}\label{l3} \int_\rn \ |\vp(y)|^q\, |y|^{\alpha\beta} \,dy&\leq
C_H\int_\rn \ |\nabla \vp(y)|^q\,|y|^{\alpha\beta+\gamma+q} \,dy\,,
\end{align} 
with classical Hardy inequality with power weights where the optimal constant is  $q^q\big(\alpha\beta+\gamma+N\big)^{-q}$. Therefore, $C_H$ cannot be improved.
\end{proof}

By the same arguments as in the proof of Example~\ref{ex:2_laplace} one can show the following consequence of Theorem~\ref{theo:inq-sgn-Lap-q}.

\begin{ex}
\label{ex:2_laplace2}
Let $q>1$, $\alpha\in\R$, $\beta>0$, ${\alpha}{\beta}+N>0$.
Then  for all compactly supported $\vp\in W^{1,\infty}(\rn)$ it holds that
 \begin{equation}
    \label{inq:ex-h-2}
\int_\rn \ |\vp|^q(1+|x|^\beta)^\alpha \,dx\leq C_{H} \int_\rn \ |\nabla \vp|^q |x|^{q}(1+|x|^\beta)^\alpha \,dx\end{equation} with a positive, finite constant $C_{H}$.
\end{ex}
Note that the above example for $q=2=\beta$ relates to the main result of \cite{bbdgv1} used there and in~\cite{bbdgv2} in the study of the asymptotics to fast diffusion equation. Indeed, within such a choice of parameters the left-hand side weight in~\eqref{inq:ex-h-2} is the same as in the inequality $(1)$ of \cite{bbdgv1}, but the right-hand weight of~\cite[(1)]{bbdgv1} reads $(1+|x|^2)^{\alpha+1}$.


\subsection{Proof of Corollary~\ref{coro:H-diff}}\label{sec:coro}
\begin{proof}We verify assumptions of Example~\ref{ex:2_laplace} applied with $q=2$, $\alpha=-\frac{p-1}{2-p},$ $\beta=\tfrac{p}{p-1}>0$, and $\gamma=-\tfrac{p}{p-1}$. We see that $\alpha\beta+\gamma+2=2-\frac{p}{(2-p)(p-1)} <0,$ $\gamma+2=2-\tfrac{p}{p-1}<0$. Since $\alpha<0$, also $|\alpha\beta+\gamma+2|\geq |\gamma+2|$. Moreover, $\alpha\beta+\gamma+N=N-\frac{p}{(2-p)(p-1)} >0$, which in particular implies that $p>N/(N-1)$ equivalent to $\gamma+N>0$. Then by Example~\ref{ex:2_laplace} we know that $C_{H}\in(0,\infty).$ The optimality of $C_H$ follows from the fact that if $N\leq 7$ or $p\not\in (p_-,p_+)$, condition~\eqref{opt-as} is satisfied. \end{proof}

\section{Long-term asymptotic behaviour of solutions to DNLE}\label{sec:as}

Before we present the proof of Theorem~\ref{theo:as}, let us recall important properties of solutions to~\eqref{dnle}. For an initial datum $u_0\in L^1(\R^N)$, existence and uniqueness  to~\eqref{dnle} is settled, see~\cite{Li2001,Agueh2005}. Moreover, in the range or parameters under consideration, we have that $u(t)\in C^{1, \alpha}(\R^N)$ for some $\alpha\in(0, 1)$, see~\cite{Vespri2020,Duzgun19,fornaro21,DiBenedetto} and mass is conserved, i.e.,  \[\int_{\R^N}u(t,x)\, dx=\int_{\R^N}u_0(x)\, dx\quad\text{ for $t>0$,}\]  
For further information about basic properties of solutions to~\eqref{dnle} we refer to the monographs~\cite{DGV-Book}, \cite[Part III]{Vazquez2007} and references therein.\newline{}

Rates of convergence like~\eqref{convergence-rate} with the use of different norms for equation~\eqref{dnle} has been a problem attracting a lot of attention. In the case $p=2$ and $m>1$, optimal rates of convergence were obtained independently by Carrillo and Toscani~\cite{Carrillo2000}, and if $m>1-1/N$ by Del Pino and Dolbeault~\cite{delpino2002a} and by Otto~\cite{Otto2001}. When $p=2$ and $1-1/N<m<1$, see also~\cite{bonforte2021stability, Bonforte2020, Carrillo2003} for rates of convergence in the stronger norm of uniform relative error. In the case of $p=2$ and $1-N/(N+2)<m<1$ improved convergence rate were obtained by Dolbeault and Toscani using the \emph{best matching} Barenblatt profile, see~\cite{Dolbeault2011,Dolbeault2013}. For $p=2$ and $0<m<1-1/N$ rates of convergence (optimal on a subrange) were computed, independently, by Carrillo and Vazquez in~\cite{Carrillo2003}, McCann and Slepcev in~\cite{McCann2006}, Kim and McCann in~\cite{Kim2006} and, for the whole range $m<1$ (negative values of $m$ are admitted) by Blanchet, Bonforte, Dolbeault, Grillo and V\'azquez in~\cite{bbdgv2}. The aforementioned results are often based on the study of a linearised problem performed in~\cite{Denzler2005}, see also the monograph~\cite{Denzler2015}.

In the case $m=1$ and $p\neq2$ less is known. In the range $2N/(N+1)+1/(N+1)\le p < N$, (non-optimal) rates of convergence were obtain by Del Pino and Dolbeault in~\cite{DelPino2002}. Similarly, in the case of the doubly nonlinear equation (non-optimal) rates of convergence were obtain by Del Pino and Dolbeault in~\cite{DelPino2003}, by Agueh~\cite{Agueh2008} and by Agueh, Blanchet and Carrillo in~\cite{Agueh2009}.

Here we only sketch the proof of Theorem~\ref{theo:as} and emphasize where Theorem~\ref{theo:HP-q} is applied. For details we refer to~\cite{Agueh2009}.

\begin{proof}[Proof of Theorem~\ref{theo:as}] We restrict our analysis to the case $p\neq2$ which is fully covered by~\cite{bbdgv2}.

By the scaling properties of the equation~\eqref{dnle}, we can assume that $D_\star=1$ and let us define $\BSD=\BS_{1}$. Let us define the \emph{self-similar} change of variables
\begin{equation}\label{self.similar}
v(\tau, y):=R(t)^N\,u(t,x)\,,\quad\mbox{where}\quad \tau:=\log\left(\RS(t)\right)\,,\quad y:=x/R(t)\,.
\end{equation}
The main advantage of the change of variables is that the Barenblatt profile $B_1(t,x)$ is transformed into $\BSD$ which is stationary in time. The equation satisfied by $v$ is now
\begin{equation}\label{rescaled-dnle}
\frac{\partial v}{\partial \tau}=\mathrm{div}\left(m^{p-1}\,v^{(1-m)(1-p)}\,|\nabla v|^{p-2}\,\nabla v+v\,y\right)\,.
\end{equation}
 The \emph{relative entropy} with respect to the Barenblatt profile $\BSD$ is defined by
\begin{equation}\label{entropy}
    \E[v|\BSD]:=\frac{m}{\sigma(\sigma-1)}\int_{\R^N}\left(v^\sigma(y)-\BSD^\sigma(y)-\sigma\,\BSD^{\sigma-1}\left(v(y)-\BSD(y)\right)\right)dy
\end{equation}
for any $0\le v\in L^1(\R^N)$ where $\sigma=m+(p-2)/(p-1)$. Notice that, by the convexity of the function $u\mapsto m\,u^\sigma/\sigma(\sigma-1)$, the relative entropy is nonnegative functional $\E[v|\BSD]\ge0$. The derivative of $\E[v(\tau),\BSD]$ along the flow~\eqref{rescaled-dnle} is called~\emph{Fisher information} and is formally given by
\begin{equation}\label{entropy-prod}\begin{split}
    \I[v|\BSD]&:=-\frac{d \E[v(\tau)|\BSD]}{d\tau}\\
    &= m^p\int_{\R^N}v(\tau)\left(\frac{\nabla v(\tau)}{v(\tau)^{2-\sigma}}-\frac{\nabla \BSD}{\BSD^{2-\sigma}}\right)\cdot\left(v(\tau)^{(\sigma-2)(p-1)}\frac{\nabla v}{|\nabla v|^{2-p}}-\BSD^{(\sigma-2)(p-1)}\frac{\nabla \BSD}{|\nabla \BSD|^{2-p}}\right)dy\,,
\end{split}\end{equation}
where $a\cdot b$ is the standard scalar product between $a,b\in\R^N$. 
Since $(a-b)\cdot(|a|^{p-2}a-|b|^{p-2}b)\ge0$ for $p>1$ and any $a,b\in\R^N$, we infer that $\I[v(\tau)|\BSD]\ge0$. 

The convergence of $v(\tau)$ to $\BSD$ as $\tau\to\infty$ follows from the decaying in time of the entropy functional. Once the inequality
\begin{equation}\label{decay-entropy}
    \E[v(\tau)|\BSD]\le e^{-\mu\,\tau}\,\E[v(0)|\BSD]\quad\forall\tau\ge\tau_0
\end{equation}
is established for some $\mu, \tau_0>0$, then inequality~\eqref{convergence-rate} can be obtained by combining~\eqref{decay-entropy} with the Csisz\'ar-Kullback inequality
\[
\|v(\tau)-\BSD\|^2_{L^1(\R^N)}\le C\,\E[v(\tau)|\BSD]\,,
\]
for which we refer to~\cite{Carrillo2001,Agueh2008}. Finally,  to obtain the algebraic rate in the $t$ variable, one should take into account the self-similar change of variables~\eqref{self.similar} which also gives the relation $\lambda=\mu/\vt$.

In order to prove~\eqref{decay-entropy}, it suffices to have
\begin{equation}\label{entropy-production-inq}
    \mu\,\E[v(\tau)|\BSD]\le \I[v(\tau)|\BSD]\,.
\end{equation}
for $\tau$ large enough. Indeed,~\eqref{decay-entropy} follows by combining a Gronwall-type argument with inequality~\eqref{entropy-production-inq} and~\eqref{entropy-prod}. In the range of parameters under consideration, inequality~\eqref{entropy-production-inq} does not hold for any function $f\in C_c^\infty(\R^N)$, due to scaling arguments. This is very different from the case, $1>m(p-1)>1-(p-1)/N$, where~\eqref{entropy-production-inq} is equivalent to a class of Gagliardo--Nirenberg--Sobolev inequalities, see~\cite{Agueh2008,bonforte2021stability,dolbeault2021functional,delpino2002a}. However, when $v(\tau)$ is close enough to $\BSD$, then~\eqref{entropy-production-inq} holds. In order to prove so, let us introduce the following weights:
\begin{equation*}
w_1(x)=\frac{1}{m}\,\left(1+\frac{(1-\sigma)\,(p-1)}{p\,m}|x|^\frac{p}{p-1}\right)^\frac{2-\sigma}{\sigma-1},\quad w_2(x)=|x|^\frac{p-2}{p-1}\,\left(1+\frac{(1-\sigma)\,(p-1)}{p\,m}|x|^\frac{p}{p-1} \right)^\frac{1}{\sigma-1}
\end{equation*}
and for any $\varepsilon\in(0,1)$
\[
w_{2,\ve}(x)=\left(1+\frac{(1-\sigma)\,(p-1)}{p\,m}|x|^\frac{p}{p-1}\right)^\frac{1}{\sigma-1}\,\left(\ve+|x|^\frac{1}{p-1}\right)^{-(2-p)}.
\]
Let us define the \emph{linearised relative entropy} as
\[
E[\vp]:=   \int_{\R^N} |\vp-\overline{\vp}|^2\, w_1(x)\,dx\]  and the \emph{linearised Fisher information} that for $p>2$ takes a form\[ I[\vp]:= \int_{\R^N} |\nabla{\vp}|^2\, w_{2}(x)\,dx\,.
\]
When $1<p<2$, the role of the \emph{linearised Fisher information} is played by the quantity $I_\ve[\vp]$  defined as $I[\vp]$ with the weight $w_{2,\ve} $ instead of $w_2$.  The inequalities
\begin{equation}\label{hardy-with-eps}
C_{p,m}^{(1)}\, E[\vp]\le  I[\vp]\qquad\mbox{and}\qquad C_{p,m}^{(2)}\,E[\vp]\le  I_\ve[\vp]
\end{equation}
hold within the range of parameters~\eqref{good.range} and for any function $\vp\in C^{1, \alpha}(\R^N)$ and $\alpha\in(0,1)$. The inequalities in~\eqref{hardy-with-eps} follow from Example~\ref{ex:hp} with $\gamma=0$, $\beta=p/(p-1)$ and $\alpha=(2-\sigma)/(\sigma-1)$ and the fact that $c_1\,w_1(|x|)\le (1+|x|^\beta)^\alpha$, $|x|^2(1+|x|^\beta)^\alpha \le c_2\,w_{2}(|x|)$ and $|x|^2(1+|x|^\beta)^\alpha \le c_{2,\ve}\,w_{2, \ve}(|x|)$, for some finite $c_1,c_2, c_{2,\ve}>0$ depending on $m,p$. 

By~\cite[Proposition 4.2]{Agueh2009}, if $p>2$ there exist $\tau_0$, $\kappa_1$ ,$  \kappa_2>0$ such that, for all $\tau>\tau_0$
\begin{equation}\label{linear-non-linear}
    I[v(\tau)-\BSD]\le \kappa_1\,\I[v(\tau)|\BSD]+\kappa_2\,E[v(\tau)-\BSD]\,.
\end{equation}
For $1<p<2$, the same inequality holds but with $I_\ve[v(\tau)-\BSD]$ instead of $I[v(\tau)-\BSD]$ in the left-hand-side. The constant $\kappa_2$ can be taken arbitrary small provided that $\tau_0$ is large enough. By~\cite[Proposition 4.1]{Agueh2009}, it holds for any $\tau>0$ that
\begin{equation}\label{linear-entropy}
    \tfrac{1}{m}\,D_1^\frac{1}{m(p-1)-1}\,\E[v(\tau)|\BSD]\le E[v(\tau)-\BSD]\,.
\end{equation}
Combining inequalities~\eqref{linear-entropy},~\eqref{linear-non-linear},~\eqref{hardy-with-eps} and provided $\tau_0$ is large enough, we find~\eqref{entropy-production-inq} with $\mu=D_1^\frac{1}{m(p-1)-1}\left(C_{p,m}^{(i)}
-\kappa_2\right)/m\kappa_1$ with $i=1$ when $p>2$ and $i=2$ when $1<p<2$.  Choosing $\tau_0$ large enough so that $C_{p,m}^{(i)}-\kappa_2\ge C_{p,m}^{(i)}/2$, we conclude that the rate of convergence $\lambda$ can be taken as
\[
\lambda=\frac{D_1^\frac{1}{m(p-1)-1}\,C_{p,m}^{(i)}}{2\,m\,\kappa_2\,(p-N[1-m(p-1)])}\,.
\]
The proof is concluded.
\end{proof}

\section*{Acknowledgements} Part of this project was carried out in Univesidad Aut\'onoma de Madrid, when IC was visiting Matteo Bonforte. Both authors are grateful to him for guidance, patience, and invaluable help.  Additionally, IC would like to thank Micha\l{} Strzelecki for insightful discussions and B\l{}a\.zej Mia-sojedow for essential help with computations.  \newline

{\copyright~2021 by the authors. This paper may be reproduced, in its entirety, for non-commercial purposes.}

\bibliographystyle{abbrv}
\bibliography{bss}

\begin{thebibliography}{10}

\bibitem{Agueh2003}
M.~Agueh.
\newblock Asymptotic behavior for doubly degenerate parabolic equations.
\newblock {\em C. R. Math. Acad. Sci. Paris}, 337(5):331--336, 2003.

\bibitem{Agueh2005}
M.~Agueh.
\newblock Existence of solutions to degenerate parabolic equations via the
  {M}onge-{K}antorovich theory.
\newblock {\em Adv. Differential Equations}, 10(3):309--360, 2005.

\bibitem{Agueh2008}
M.~Agueh.
\newblock Rates of decay to equilibria for {$p$}-{L}aplacian type equations.
\newblock {\em Nonlinear Anal.}, 68(7):1909--1927, 2008.

\bibitem{Agueh2009}
M.~{Agueh}, A.~{Blanchet}, and J.~A. {Carrillo}.
\newblock {Large time asymptotics of the doubly nonlinear equation in the
  non-displacement convexity regime}.
\newblock {\em {J. Evol. Equ.}}, 10(1):59--84, 2010.

\bibitem{andrews99}
G.~E. Andrews, R.~Askey, and R.~Roy.
\newblock {\em Special functions}, volume~71 of {\em Encyclopedia of
  Mathematics and its Applications}.
\newblock Cambridge University Press, Cambridge, 1999.

\bibitem{sob-log-overview}
C.~An\'{e}, S.~Blach\`ere, D.~Chafa\"{\i}, P.~Foug\`eres, I.~Gentil,
  F.~Malrieu, C.~Roberto, and G.~Scheffer.
\newblock {\em Sur les in\'{e}galit\'{e}s de {S}obolev logarithmiques},
  volume~10 of {\em Panoramas et Synth\`eses [Panoramas and Syntheses]}.
\newblock Soci\'{e}t\'{e} Math\'{e}matique de France, Paris, 2000.
\newblock With a preface by Dominique Bakry and Michel Ledoux.

\bibitem{a1970}
D.~G. Aronson.
\newblock Regularity properties of flows through porous media: {T}he interface.
\newblock {\em Arch. Rational Mech. Anal.}, 37:1--10, 1970.

\bibitem{BaGo}
P.~Baras and J.~A. Goldstein.
\newblock The heat equation with a singular potential.
\newblock {\em Trans. Amer. Math. Soc.}, 284(1):121--139, 1984.

\bibitem{BFT2004}
G.~Barbatis, S.~Filippas, and A.~Tertikas.
\newblock A unified approach to improved {$L^p$} {H}ardy inequalities with best
  constants.
\newblock {\em Trans. Amer. Math. Soc.}, 356(6):2169--2196, 2004.

\bibitem{BR}
F.~Barthe and C.~Roberto.
\newblock Modified logarithmic {S}obolev inequalities on {$\mathbb{ R}$}.
\newblock {\em Potential Anal.}, 29(2):167--193, 2008.

\bibitem{bbdgv1}
A.~Blanchet, M.~Bonforte, J.~Dolbeault, G.~Grillo, and J.-L. V\'azquez.
\newblock Hardy-{P}oincar\'e inequalities and applications to nonlinear
  diffusions.
\newblock {\em C. R. Math. Acad. Sci. Paris}, 344(7):431--436, 2007.

\bibitem{bbdgv2}
A.~Blanchet, M.~Bonforte, J.~Dolbeault, G.~Grillo, and J.~L. V\'{a}zquez.
\newblock Asymptotics of the fast diffusion equation via entropy estimates.
\newblock {\em Arch. Ration. Mech. Anal.}, 191(2):347--385, 2009.

\bibitem{Bobkov99}
S.~G. Bobkov and F.~G\"{o}tze.
\newblock Exponential integrability and transportation cost related to
  logarithmic {S}obolev inequalities.
\newblock {\em J. Funct. Anal.}, 163(1):1--28, 1999.

\bibitem{BDM2018}
V.~B\"{o}gelein, F.~Duzaar, P.~Marcellini, and C.~Scheven.
\newblock Doubly nonlinear equations of porous medium type.
\newblock {\em Arch. Ration. Mech. Anal.}, 229(2):503--545, 2018.

\bibitem{bonforte2021stability}
M.~Bonforte, J.~Dolbeault, B.~Nazaret, and N.~Simonov.
\newblock Stability in {G}agliardo--{N}irenberg--{S}obolev inequalities: flows,
  regularity and the entropy method, \texttt{https://arxiv.org/abs/2007.03674},
  2021.

\bibitem{Bonforte2009}
M.~Bonforte, G.~Grillo, and J.~L. V{\'{a}}zquez.
\newblock Special fast diffusion with slow asymptotics: Entropy method and flow
  on a riemannian manifold.
\newblock {\em Arch. Ration. Mech. Anal.}, 196(2):631--680, 2009.

\bibitem{Bonforte2020}
M.~Bonforte and N.~Simonov.
\newblock Fine properties of solutions to the cauchy problem for a fast
  diffusion equation with {C}affarelli--{K}ohn--{N}irenberg weights,
  \texttt{https://arxiv.org/abs/2002.09967}, 2020.

\bibitem{Bonforte2021}
M.~Bonforte, N.~Simonov, and D.~Stan.
\newblock The {C}auchy problem for the fast $p-${L}aplacian evolution equation.
  {C}haracterization of the global {H}arnack principle and fine asymptotic
  behaviour.
\newblock 2021.

\bibitem{Carrillo2001}
J.~A. Carrillo, A.~J\"{u}ngel, P.~A. Markowich, G.~Toscani, and A.~Unterreiter.
\newblock Entropy dissipation methods for degenerate parabolic problems and
  generalized {S}obolev inequalities.
\newblock {\em Monatsh. Math.}, 133(1):1--82, 2001.

\bibitem{CLM2002}
J.~A. Carrillo, C.~Lederman, P.~A. Markowich, and G.~Toscani.
\newblock Poincar\'{e} inequalities for linearizations of very fast diffusion
  equations.
\newblock {\em Nonlinearity}, 15(3):565--580, 2002.

\bibitem{Carrillo2000}
J.~A. Carrillo and G.~Toscani.
\newblock Asymptotic {$L^1$}-decay of solutions of the porous medium equation
  to self-similarity.
\newblock {\em Indiana Univ. Math. J.}, 49(1):113--142, 2000.

\bibitem{Carrillo2003}
J.~A. Carrillo and J.~L. V\'{a}zquez.
\newblock Fine asymptotics for fast diffusion equations.
\newblock {\em Comm. Partial Differential Equations}, 28(5-6):1023--1056, 2003.

\bibitem{sph-Poin-constant-book}
I.~Chavel.
\newblock {\em Eigenvalues in {R}iemannian geometry}, volume 115 of {\em Pure
  and Applied Mathematics}.
\newblock Academic Press, Inc., Orlando, FL, 1984.
\newblock Including a chapter by Burton Randol, With an appendix by Jozef
  Dodziuk.

\bibitem{ChZa}
I.~Chlebicka and A.~Zatorska-Goldstein.
\newblock Existence to nonlinear parabolic problems with unbounded weights.
\newblock {\em J. Evol. Equ.}, 19(1):1--19, 2019.

\bibitem{Chua}
S.-K. Chua.
\newblock On weighted {S}obolev interpolation inequalities.
\newblock {\em Proc. Amer. Math. Soc.}, 121(2):441--449, 1994.

\bibitem{DA2}
L.~D'Ambrosio.
\newblock Hardy inequalities related to {G}rushin type operators.
\newblock {\em Proc. Amer. Math. Soc.}, 132(3):725--734, 2004.

\bibitem{DA}
L.~D'Ambrosio.
\newblock Hardy-type inequalities related to degenerate elliptic differential
  operators.
\newblock {\em Ann. Sc. Norm. Super. Pisa Cl. Sci. (5)}, 4(3):451--486, 2005.

\bibitem{DA3}
L.~D'Ambrosio and S.~Dipierro.
\newblock Hardy inequalities on {R}iemannian manifolds and applications.
\newblock {\em Ann. Inst. H. Poincar\'{e} Anal. Non Lin\'{e}aire},
  31(3):449--475, 2014.

\bibitem{delpino2002a}
M.~Del~Pino and J.~Dolbeault.
\newblock Best constants for {G}agliardo-{N}irenberg inequalities and
  applications to nonlinear diffusions.
\newblock {\em J. Math. Pures Appl. (9)}, 81(9):847--875, 2002.

\bibitem{DelPino2002}
M.~Del~Pino and J.~Dolbeault.
\newblock Nonlinear diffusions and optimal constants in {S}obolev type
  inequalities: asymptotic behaviour of equations involving the
  {$p$}-{L}aplacian.
\newblock {\em C. R. Math. Acad. Sci. Paris}, 334(5):365--370, 2002.

\bibitem{DelPino2003}
M.~Del~Pino and J.~Dolbeault.
\newblock Asymptotic behavior of nonlinear diffusions.
\newblock {\em Mathematical Research Letters}, 10(4):551--557, 2003.

\bibitem{Denzler2015}
J.~Denzler, H.~Koch, and R.~J. McCann.
\newblock Higher-order time asymptotics of fast diffusion in {E}uclidean space:
  a dynamical systems approach.
\newblock {\em Memoirs of the American Mathematical Society}, 234(1101):vi+81,
  2015.

\bibitem{Denzler2005}
J.~Denzler and R.~J. McCann.
\newblock Fast diffusion to self-similarity: complete spectrum, long-time
  asymptotics, and numerology.
\newblock {\em Arch. Ration. Mech. Anal.}, 175(3):301--342, 2005.

\bibitem{DiBenedetto}
E.~{DiBenedetto}.
\newblock {\em Degenerate parabolic equations}.
\newblock New York, NY: Springer-Verlag, 1993.

\bibitem{DGV-Book}
E.~DiBenedetto, U.~Gianazza, and V.~Vespri.
\newblock {\em Harnack{\textquotesingle}s Inequality for Degenerate and
  Singular Parabolic Equations}.
\newblock Springer New York, 2012.

\bibitem{dolbeault2021functional}
J.~Dolbeault.
\newblock Functional inequalities: nonlinear flows and entropy methods as a
  tool for obtaining sharp and constructive results,
  \texttt{https://arxiv.org/abs/2107.08219}, 2021.

\bibitem{Dolbeault2011}
J.~Dolbeault and G.~Toscani.
\newblock Fast diffusion equations: matching large time asymptotics by relative
  entropy methods.
\newblock {\em Kinetic and Related Models}, 4(3):701--716, 2011.

\bibitem{Dolbeault2013}
J.~Dolbeault and G.~Toscani.
\newblock Improved interpolation inequalities, relative entropy and fast
  diffusion equations.
\newblock {\em Ann. Inst. H. Poincar\'{e} Anal. Non Lin\'{e}aire},
  30(5):917--934, 2013.

\bibitem{Duzgun19}
F.~G. D\"{u}zg\"{u}n, S.~Mosconi, and V.~Vespri.
\newblock Harnack and pointwise estimates for degenerate or singular parabolic
  equations.
\newblock In {\em Contemporary research in elliptic {PDE}s and related topics},
  volume~33 of {\em Springer INdAM Ser.}, pages 301--368. Springer, Cham, 2019.

\bibitem{fornaro21}
S.~Fornaro, E.~Henriques, and V.~Vespri.
\newblock Regularity results for a class of doubly nonlinear very singular
  parabolic equations.
\newblock {\em Nonlinear Anal.}, 205:112213, 30, 2021.

\bibitem{FSV2015}
S.~Fornaro, M.~Sosio, and V.~Vespri.
\newblock Harnack type inequalities for some doubly nonlinear singular
  parabolic equations.
\newblock {\em Discrete Contin. Dyn. Syst.}, 35(12):5909--5926, 2015.

\bibitem{GaPa}
J.~P. Garc\'{\i}a~Azorero and I.~Peral~Alonso.
\newblock Hardy inequalities and some critical elliptic and parabolic problems.
\newblock {\em J. Differential Equations}, 144(2):441--476, 1998.

\bibitem{Ghoussoub2013}
N.~Ghoussoub.
\newblock {\em Functional inequalities : new perspectives and new
  applications}.
\newblock American Mathematical Society, Providence, Rhode Island, 2013.

\bibitem{GM}
N.~Ghoussoub and A.~Moradifam.
\newblock Bessel pairs and optimal {H}ardy and {H}ardy-{R}ellich inequalities.
\newblock {\em Math. Ann.}, 349(1):1--57, 2011.

\bibitem{GuW}
C.~E. Guti\'{e}rrez and R.~L. Wheeden.
\newblock Sobolev interpolation inequalities with weights.
\newblock {\em Trans. Amer. Math. Soc.}, 323(1):263--281, 1991.

\bibitem{hebey}
E.~Hebey.
\newblock {\em Nonlinear analysis on manifolds: {S}obolev spaces and
  inequalities}, volume~5 of {\em Courant Lecture Notes in Mathematics}.
\newblock New York University, Courant Institute of Mathematical Sciences, New
  York; American Mathematical Society, Providence, RI, 1999.

\bibitem{Huang2021}
X.~Huang and D.~Ye.
\newblock First order {H}ardy inequalities revisited,
  \texttt{http\://arxiv.org/pdf/2109.05471v1:PDF}.
\newblock 2021.

\bibitem{hutter97}
K.~Hutter.
\newblock Mathematical foundation of ice sheet and ice shelf dynamics. {A}
  physicist's view.
\newblock In {\em Free boundary problems: theory and applications ({C}rete,
  1997)}, volume 409 of {\em Chapman \& Hall/CRC Res. Notes Math.}, pages
  192--203. Chapman \& Hall/CRC, Boca Raton, FL, 1999.

\bibitem{IMY1994}
A.~V. Ivanov, P.~Z. Mkrtychyan, and V.~Yaeger.
\newblock Existence and uniqueness of a regular solution of the first
  initial-boundary value problem for a class of doubly nonlinear parabolic
  equations.
\newblock {\em Zap. Nauchn. Sem. S.-Peterburg. Otdel. Mat. Inst. Steklov.
  (POMI)}, 213(Kraev. Zadachi Mat. Fiz. Smezh. Voprosy Teor. Funktsi\u{\i}.
  25):48--65, 224--225, 1994.

\bibitem{KPP}
A.~Ka\l{}amajska and K.~Pietruska-Pa\l{}uba.
\newblock On a variant of the {G}agliardo-{N}irenberg inequality deduced from
  the {H}ardy inequality.
\newblock {\em Bull. Pol. Acad. Sci. Math.}, 59(2):133--149, 2011.

\bibitem{Kalashnikov1987}
A.~S. Kalashnikov.
\newblock Some problems of the qualitative theory of non-linear degenerate
  second-order parabolic equations.
\newblock {\em Russian Mathematical Surveys}, 42(2):169--222, 1987.

\bibitem{Kim2006}
Y.~J. Kim and R.~J. McCann.
\newblock Potential theory and optimal convergence rates in fast nonlinear
  diffusion.
\newblock {\em J. Math. Pures Appl. (9)}, 86(1):42--67, 2006.

\bibitem{lad69}
O.~A. Ladyzhenskaya.
\newblock {\em The mathematical theory of viscous incompressible flow}.
\newblock Mathematics and its Applications, Vol. 2. Gordon and Breach, Science
  Publishers, New York-London-Paris, 1969.
\newblock Second English edition, revised and enlarged, Translated from the
  Russian by Richard A. Silverman and John Chu.

\bibitem{Li2001}
J.~Li.
\newblock Cauchy problem and initial trace for a doubly degenerate parabolic
  equation with strongly nonlinear sources.
\newblock {\em J. Math. Anal. Appl.}, 264(1):49--67, 2001.

\bibitem{McCann2006}
R.~J. McCann and D.~Slep\v{c}ev.
\newblock Second-order asymptotics for the fast-diffusion equation.
\newblock {\em International Mathematics Research Notices}, pages Art. ID
  24947, 22, 2006.

\bibitem{Miclo}
L.~Miclo.
\newblock Quand est-ce que des bornes de {H}ardy permettent de calculer une
  constante de {P}oincar\'{e} exacte sur la droite?
\newblock {\em Ann. Fac. Sci. Toulouse Math. (6)}, 17(1):121--192, 2008.

\bibitem{PoMi}
E.~Mitidieri and S.~I. Pokhozhaev.
\newblock Absence of positive solutions for quasilinear elliptic problems in
  {${\bf R}^N$}.
\newblock {\em Tr. Mat. Inst. Steklova}, 227(Issled. po Teor. Differ. Funkts.
  Mnogikh Perem. i ee Prilozh. 18):192--222, 1999.

\bibitem{M}
B.~Muckenhoupt.
\newblock Hardy's inequality with weights.
\newblock {\em Studia Math.}, 44:31--38, 1972.

\bibitem{Otto2001}
F.~Otto.
\newblock The geometry of dissipative evolution equations: the porous medium
  equation.
\newblock {\em Comm. Partial Differential Equations}, 26(1-2):101--174, 2001.

\bibitem{S2020}
L.~Sch\"{a}tzler.
\newblock The obstacle problem for singular doubly nonlinear equations of
  porous medium type.
\newblock {\em Atti Accad. Naz. Lincei Rend. Lincei Mat. Appl.},
  31(3):503--548, 2020.

\bibitem{S1}
I.~Skrzypczak.
\newblock Hardy-type inequalities derived from {$p$}-harmonic problems.
\newblock {\em Nonlinear Anal.}, 93:30--50, 2013.

\bibitem{S2}
I.~Skrzypczak.
\newblock Hardy-{P}oincar\'{e} type inequalities derived from {$p$}-harmonic
  problems.
\newblock In {\em Calculus of variations and {PDE}s}, volume 101 of {\em Banach
  Center Publ.}, pages 225--238. Polish Acad. Sci. Inst. Math., Warsaw, 2014.

\bibitem{JLVSmoothing}
J.~L. V\'{a}zquez.
\newblock {\em Smoothing and decay estimates for nonlinear diffusion equations,
  \rm{Equations of porous medium type}}, volume~33 of {\em Oxford Lecture
  Series in Mathematics and its Applications}.
\newblock Oxford University Press, Oxford, 2006.

\bibitem{Vazquez2007}
J.~L. V\'{a}zquez.
\newblock {\em The porous medium equation, \rm{Mathematical theory}}.
\newblock Oxford Mathematical Monographs. The Clarendon Press, Oxford
  University Press, Oxford, 2007.

\bibitem{Vaz17}
J.~L. V{\'{a}}zquez.
\newblock The mathematical theories of diffusion: Nonlinear and fractional
  diffusion.
\newblock In {\em Nonlocal and Nonlinear Diffusions and Interactions: New
  Methods and Directions}, pages 205--278. Springer International Publishing,
  2017.

\bibitem{VaZu}
J.~L. V\'{a}zquez and E.~Zuazua.
\newblock The {H}ardy inequality and the asymptotic behaviour of the heat
  equation with an inverse-square potential.
\newblock {\em J. Funct. Anal.}, 173(1):103--153, 2000.

\bibitem{Vespri2020}
V.~Vespri and M.~Vestberg.
\newblock An extensive study of the regularity of solutions to doubly singular
  equations.
\newblock {\em Advances in Calculus of Variations}, 2021.

\bibitem{Zhuoqun2001}
Z.~{Wu}, J.~{Zhao}, J.~{Yin}, and H.~{Li}.
\newblock {\em Nonlinear diffusion equations}.
\newblock Singapore: World Scientific, 2001.

\end{thebibliography}

\end{document}